%% file: TensorSandwich.tex
\documentclass[conference] {IEEEtran}
\usepackage[utf8]{inputenc}
\usepackage[reqno]{amsmath}
\usepackage{amssymb}
\usepackage{amsthm}
\usepackage{amscd}
\usepackage{comment}
\usepackage{enumerate}
\usepackage{setspace,hyperref}
\usepackage[algo2e]{algorithm2e} 
\newcommand{\norm}[2]{\left\|#1\right\|_{#2}}
\newcommand{\vb}[1]{\mathbf{#1}}
\newcommand{\col}{\text{col-span}}

\input{macro}
\numberwithin{equation}{section}
\begin{document}

\title{Tensor Sandwich: Tensor Completion for Low CP-Rank Tensors via Adaptive Random Sampling}

\author{
\IEEEauthorblockN{Cullen Haselby, Santhosh Karnik, Mark Iwen}

\IEEEauthorblockA{Michigan State University, \\\
{\tt haselbyc@msu.edu, karniksa@msu.edu, iwenmark@msu.edu}}
}
\maketitle

\begin{abstract}
We propose an adaptive and provably accurate tensor completion approach based on combining matrix completion techniques (see, e.g., \cite{candes2012exact,krishnamurthy2014power,Ward2015}) for a small number of slices with a modified noise robust version of Jennrich’s algorithm (see, e.g., \cite[Section 3]{moitra2018algorithmic}).  In the simplest case, this leads to a sampling strategy that more densely samples two outer slices (the bread), and then more sparsely samples additional inner slices (the bbq-braised tofu) for the final completion. Under mild assumptions on the factor matrices, the proposed algorithm completes an $n \times n \times n$ tensor with CP-rank $r$ with high probability while using at most $\mathcal{O}(nr\log^2 r)$ adaptively chosen samples.  Empirical experiments further verify that the proposed approach works well in practice, including as a low-rank approximation method in the presence of additive noise.
\end{abstract}

\section{Introduction}

Consider a CP rank-$r$ tensor $\calT \in \R^{n\times n\times n}$ where $r \le n$ with CP-decomposition 
\begin{equation}
\label{equ:ExactLowCPRank}
\calT = \sum_{i=1}^r \vb{a}_i \circ \vb{b}_i \circ \vb{c}_i.
\end{equation}
Here $\vb{a}_i, \vb{b}_i, \vb{c}_i \in \R^n$ for all $i \in [r] := \{1, \dots, r \}$, and $\circ$ denotes the outer product so that $\vb{a}_i \circ \vb{b}_i \circ \vb{c}_i \in \R^{n\times n\times n}$ has entries $\left( \vb{a}_i \circ \vb{b}_i \circ \vb{c}_i \right)_{h,j,k} = (\vb{a}_i)_h (\vb{b}_i)_j (\vb{c}_i)_k$ for all $i \in [r]$.  The tensor completion problem is to reconstruct $\calT$ after observing a subset of its entries $\calT_{i,j,k}$ for $(i,j,k) \in \Omega \subset [n]^3$.  In this paper $\Omega$ will consist of random entries adaptively sampled in several rounds (i.e., so that the entries observed in round $\ell+1$ are selected via a distribution that depends on the entries observed in rounds $1$ to $\ell$).  We  hasten to add, however, that the general proof approach taken herein can itself be adapted to also yield tensor completion results based on non-adaptive sampling at the expense of further restricting the class of tensors considered below for which the methods will be guaranteed to succeed. 

In order to define the class of tensors for which the proposed approach will be guaranteed to succeed we need several definitions.  Define the \emph{factor matrix} $\mA = \begin{bmatrix} \vb{a}_1 & \vb{a}_2 & \cdots & \vb{a}_r \end{bmatrix} \in \R^{n \times r}$ to have columns given by the factor vectors ${\bf a_j}$ in \eqref{equ:ExactLowCPRank}, as well as for $\mB, \mC \in \R^{n \times r}$ with columns arranged in the same corresponding orders. We will employ the \emph{Khatri-Rao product} of two matrices; which is the matrix that results from computing the Kronecker product of their matching columns. That is, for $\mA,\mB \in \R^{n \times r}$, their Khatri-Rao product is the matrix $\mA \odot \mB \in \R^{n^2 \times r}$ defined by
\[
\mA\odot \mB := \begin{bmatrix}
\vb{a}_1 \otimes \vb{b}_1 &\vb{a}_2 \otimes \vb{b}_2 & \dots& \vb{a}_n \otimes \vb{b}_n 
\end{bmatrix}. \] This operation is useful when considering the matricizations or flattenings of a rank-$r$ CP tensor \eqref{equ:ExactLowCPRank}. E.g. the mode-3 unfolding of $\mathcal{T}$ is equal to $\mT_{(3)} = \mC (\mA \odot \mB)^T$, see \cite{kolda2009tensor}. 
Define the \emph{coherence} of an $r$-dimensional subspace $U \subset \R^n$ to be $$\mu(U) := \dfrac{n}{r}\max_{i \in [n]}\|\calP_{U}\ve_i\|_2^2,$$ where $\calP_{U}$ is the orthogonal projection onto $U$, and where $\ve_i \in \R^n$ for $i \in [n]$ denotes the $i^{\rm th}$ standard basis vector.  Also, the \emph{Kruskal rank of a matrix} is the maximum integer $r$ such that any $r$ columns of the matrix are linearly independent.  

Finally, for any positive integers $n, r, s$ with $r,s \le n$ and any $\mu_0 \in [1,n/r]$, we define $\mathbb{T}(n,r,\mu_0,s)$ to be the class of all tensors $\mathcal{T} = \sum_{i=1}^r \vb{a}_i \circ \vb{b}_i \circ \vb{c}_i \in \R^{n \times n \times n}$ for which: 
\begin{enumerate}[(a)]
    \item the factor matrices $\vb{A},\vb{B}$ both have full column rank,
    \label{assump_a}
    \item the column space of $\vb{A}$ has coherence bounded above by $\mu_0$, and
    \label{assump_b}
    \item for some $s \in [n]$, every $s \times r$ submatrix of $\vb{C}$ has Kruskal rank $\ge 2$.
    \label{assump_c}
\end{enumerate}

We are now able to state our main result.

\begin{theorem}
\label{thm:MainResult}
There exists an adaptive random sampling strategy and an associated reconstruction algorithm (see Algorithm~\ref{alg:tensor_sandwich}) such that for any $\delta > 0$, after observing at most $C_1 s\mu_0nr \log^2(r^2/\delta)$ entries of a tensor $\mathcal{T} \in \mathbb{T}(n,r,\mu_0,s)$, the algorithm completes $\mathcal{T}$ with probability at least $1-s\delta$.  Here $C_1 > 0$ is absolute constant that is independent of all other quantities.
\end{theorem}

The proof of our sandwich sampling algorithm involves three stages. First, w.l.o.g., we pick $s$ mode-$3$ slices and then use a matrix completion algorithm that with high probability recovers these slices and observes at most $C'_1 \mu_0nr \log^2(r^2/\delta)$ entries in each slice. Second, if this $n \times n \times s$ subtensor is correctly completed, we can then use a deterministic variant of Jennrich's Algorithm\footnote{As discussed in \cite{Kolda_onJennrich} the work of \cite{Leurgans1993} is perhaps a more accurate attribution of the method, however we use the traditional name of Jennrich's Algorithm in this work.} on the completed subtensor to learn the factor matrices $\vb{A}$ and $\vb{B}$. Third, once we know the factor matrices $\vb{A}$ and $\vb{B}$, we can deterministically find $r$ sample locations in each of the $n$ mode-$3$ slices whose values allow a censored least squares problem to solve for the third factor matrix $\vb{C}$. This three-stage procedure uses at most $C'_1 s\mu_0nr \log^2(r^2/\delta)$ samples to complete the $s$ initial mode-$3$ slices in order to learn $\vb{A}$ and $\vb{B}$, and then $nr$ additional samples to learn $\vb{C}$ thereafter, for a total of at most $C'_1 s\mu_0nr \log^2(r^2/\delta) + nr \le C_1 s\mu_0nr \log^2(r^2/\delta)$ samples. See Figure \ref{fig:schematic} for a schematic illustration of the overall sampling strategy where fibers are sampled through the middle of the sandwich for simplicity.

We note that assumptions \ref{assump_a} and \ref{assump_c} are the necessary assumptions for the Jennrich's step to work with any $n \times n \times s$ subtensor. Furthermore, if the columns of the factor matrices are drawn from any continuous distribution, assumption \ref{assump_c} for $s = 2$ holds with probability $1$.  Finally, something akin to assumption \ref{assump_b} is always required in completion problems.

\begin{figure}[h]
\centering
\includegraphics[width=\columnwidth]{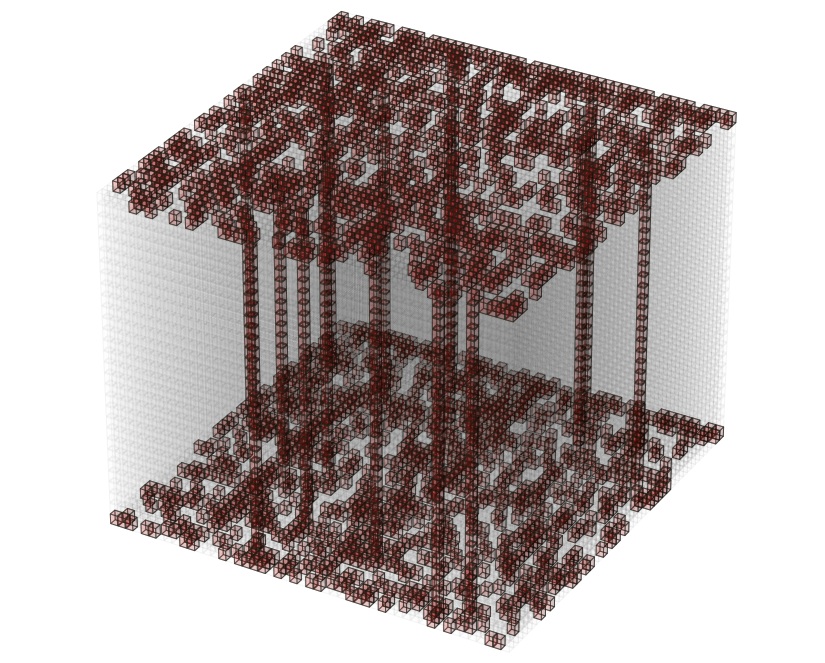}
\caption{A schematic depiction of the sampling strategy where $s = 2$ slices have been sampled relatively densely in order to compute $\vb{A}$ and $\vb{B}$, and where additional fibers where then sampled elsewhere to help compute $\vb{C}$}
\label{fig:schematic}
\end{figure}

\subsection{Related Work}
Many prior works on low-rank tensor completion use non-adaptive and uniform sampling \cite{jain2014provable,barak2016noisy,yuan2016tensor,yuan2017incoherent,potechin2017exact,montanari2018spectral,liu2020tensor,kivva2020exact}. While some of those works \cite{barak2016noisy,kivva2020exact,montanari2018spectral} can handle CP-ranks up to roughly $n^{3/2}$ instead of $n$, all of them require at least $\mathcal{O}(n^{3/2})$ samples, even when the rank is $r = \mathcal{O}(1)$. Furthermore, \cite{barak2016noisy} shows that completing an $n \times n \times n$ rank-$r$ tensor from $n^{3/2-\epsilon}$ uniformly random samples is NP-hard by comparison to the problem of refuting a 3-SAT formula with $n$ variables and $n^{3/2-\eps}$ clauses.  

In \cite{krishnamurthy2013low}, the authors propose a method for completing CP-rank $r \le n$ tensors using adaptive sampling which, for order-$3$ tensors, requires $\mathcal{O}(\mu_0^2 nr^{5/2}\log^2 r)$ samples. Our algorithm requires a number of samples which has a more favorable dependence on the coherence $\mu_0$ and rank $r$. Furthermore, our result only requires coherence assumptions about $\vb{A}$, instead of about $\vb{A}$ \emph{and} $\vb{B}$. These improvements come at the expense of requiring the mild additional assumptions \ref{assump_a} and \ref{assump_c} in our Theorem \ref{thm:MainResult} related to Jennrich's algorithm, however.

The first step in our tensor completion algorithm involves using an adaptive sampling algorithm to complete $s$ mode-$3$ slices of the tensor. Our tensor completion algorithm and results are based on the adaptive matrix completion algorithm and results in \cite{krishnamurthy2014power} which, with high probability, uses $\mathcal{O}(\mu_0 nr\log^2 r)$ samples to complete a rank $r$ matrix.  However, our algorithm can be adapted to use other adaptive matrix completion results such as, e.g., \cite{Ward2015} with relative ease.

In the censored least squares phase of our algorithm, we can sample entire fibers of the tensor as done in Figure~\ref{fig:schematic}.  We note that doing so is similar in spirit to the fiber sampling approach of Sørensen and De Lathauwer \cite{sorensen2019fiber}. However, their work focuses on determining algebraic constraints on the factor matrices of a low rank tensor which, when satisfied, allow the tensor to be completed from the sampled fibers. As such, our results cannot be directly compared with \cite{sorensen2019fiber}.

\section{Proof of Theorem~\ref{thm:MainResult}}

In this section we follow our 3 stage proof outline.

\subsection{Completing $s$ mode-$3$ slices of $\mathcal{T}$}

We start by picking any subset of indices $S \subset [n]$ with $|S| = s$ elements. For each $k \in S$, the mode-$3$ slice $\calT_{:,:,k} \in \R^{n \times n}$ satisfies $$\calT_{:,:,k} = \sum_{i = 1}^{r}\langle\vb{c}_i,\vb{e}_k\rangle\vb{a}_i\vb{b}_i^T,$$ and so, $\col(\calT_{:,:,k}) \subseteq \col(\vb{A})$. By assumption, 
$\col(\vb{A})$ has coherence bounded by $\mu_0$. Thus, the assumptions required by Theorem 1 in \cite{krishnamurthy2014power} hold. Therefore, for each slice $T_{:,:,k}$ $k \in S$, with probability at least $1-\delta$, the adaptive sampling procedure in Algorithm 1 uses at most $C'_1\mu_0 nr\log^2(r^2/\delta)$ samples for some absolute constant $C'_1 > 0$, and completes $\calT_{:,:,k}$.  

By taking a simple union bound over each of the $s$ slices $k \in S$, we have that with probability at least $1-s\delta$, this strategy will successfully complete all of the $s$ slices $\calT_{:,:,k}$ for $k \in S$ and use fewer than $C'_1s\mu_0 nr\log^2(r^2/\delta)$ samples. Let $\Omega_1 = \left\{(i,j,k) \lvert (i,j) \text{ sampled according to \cite{krishnamurthy2014power} for slice } k\in S \right\}$, the set of locations of $\mathcal{T}$ sampled to complete these $s$ slices.

\subsection{Learning mode-$1$ and $2$ factor matrices via Modified Jennrich's Algorithm / simultaneous diagonlization}

Let $\vb{u},\vb{v}$ be random vectors uniformly drawn from the unit sphere $\mathbb{S}^{s-1}$. Denote the sub-vector of ${\vb c}$ with entries indexed in $S$ by $\vb{\tilde{c}}_i = (\vb{c}_i)_S$, and construct two auxiliary matrices $ \vb{T}_{u},  \vb{T}_{v}$ using the completed slices by adding up the linear combinations of the completed slices, weighted by the random vectors $\vb{u},\vb{v}$. In terms of the components, we have:
\begin{align*}
 \vb{T}_{u} &= \sum_{i=1}^r \langle \vb{\tilde{c}}_i, \vb{u} \rangle \vb{a}_i\vb{b}_i^T  \\
   \vb{T}_{v} &= \sum_{i=1}^r \langle \vb{\tilde{c}}_i, \vb{v} \rangle \vb{a}_i\vb{b}_i^T  
\end{align*}   
Denote the $r\times r$ matrix $\vb{D}_{u}$ which has along its diagonal entries the values $\langle \vb{\tilde{c}}_i, \vb{u} \rangle $, and similarly $\vb{D}_v$. Notice the following identity with the product  $\vb{T}_{u} (\vb{T}_{v})^{\dagger}$
\begin{align*}
 \vb{T}_{u} (\vb{T}_{v})^{\dagger} &= \vb{A} \vb{D}_{u} \vb{B}^T \left(\vb{A} \vb{D}_{v} \vb{B}^T\right)^{\dagger}  \\
 &= \vb{A} \vb{D}_{u} \vb{D}_{v}^{-1} \vb{A}^{\dagger}  \\
\end{align*}    
Where the matrix $ \vb{D}_{u} \vb{D}_{v}^{-1} $ is a diagonal matrix with $(i,i)$-entry equal to $\frac{\langle \vb{\tilde{c}}_i, \vb{u} \rangle}{\langle \vb{\tilde{c}}_i, \vb{v} \rangle}$. Clearly the matrix $ \vb{T}_{u} (\vb{T}_{v})^{\dagger}$ is diagonalizable, and so by computing the eigen-decomposition of $\vb{T}_{u} (\vb{T}_{v})^{\dagger}$ we recover the columns of $\vb{A}$ and the eigenvlaues along the diagonal of $ \vb{D}_{u} \vb{D}_{v}^{-1} $. We can order the eigenvectors in descending order by magnitude of their corresponding eigenvalue. Note it is here that we employ assumption \ref{assump_c}: Since $\vb{C}_S$ (i.e., the rows of $\vb{C}$ indexed by $S$) has $k$-rank at least 2, the ratios $\frac{\langle \vb{\tilde{c}}_i, \vb{u} \rangle}{\langle \vb{\tilde{c}}_i, \vb{v} \rangle}$ for each $i$ will almost surely be distinct from one another, and thus the ordering of the eigenvalues is unique. 

Let $\vb{P}$ be the permutation matrix that interchanges the columns of $\vb{A}$ so that instead of being in $\vb{D}_{u} \vb{D}_{v}^{-1} $ order, they are in $\vb{D}_u$ order (ordered greatest to least in terms of the magnitude of $\langle \vb{\tilde{c}}_i, \vb{u} \rangle$). 
Now notice that 
\begin{equation}
\label{eqn:findB}
\begin{split}
 \vb{A}^{\dagger} \vb{T}_{u}  &= \vb{A}^{\dagger} \vb{A} \vb{P} \vb{D}_{u}\vb{B}^T \\
 &= \vb{P} \vb{D}_{u} \vb{B}^T  \\
  &=  (\vb{B}\vb{D}_u \vb{P})^T.  \\
\end{split}    
\end{equation}
This means that the rescaled columns of $\vb{B}$ that are in $\vb{D}_u$ order are now in  $\vb{D}_{u} \vb{D}_{v}^{-1} $ order after we apply the inverse $\vb{A}^{\dagger}$ to $ \vb{T}_{u} $. That is, we have found the matching components of $\vb{A}$ and $\vb{B}$ up to a re-scaling of their outer product!  This was achieved by learning the columns of $\vb{A}$ from $\vb{T}_{u} (\vb{T}_{v})^{\dagger}$ and then, crucially, the matching columns of $\vb{B}$ from $ \vb{A}^{\dagger} \vb{T}_{u} $. The scaling due to $\vb{D}_u$ will be resolved in the final step of the algorithm, where we solve for the missing third factor (i.e., $\vb{B},\vb{C}$ and $\vb{BD},\vb{C}(\vb{D}^{-1})$ are both valid pairs of factors for any diagonal matrix $\vb{D}$). 

\subsection{Learning the mode-$3$ factor matrix}
After obtaining $\vb{A}$ and a rescaled $\vb{B}$ in order to find the remaining components, $\vb{C}$, we will need $\Omega_2$, a second set of revealed locations of the entries of $\mathcal{T}$.  We will also need the solutions to $n$ instances of the following censored least squares problem related to those revealed values

\begin{equation}
   \label{eqn:censored_lstsq}
(\vb{A} \odot \vb{B})_{K_k} \vb{c} = \vb{t}_{K_k} 
\end{equation}
where $k\in[n]$,$\vb{c}=(\vb{C}_{k,:})^T$, $\vb{t}=\text{vec}(\mathcal{T}_{:,:,k})^T$ and $K_k = \left\{ i+n(j-1)\vert (i,j,k)\in\Omega_2 \right\}$. In \eqref{eqn:censored_lstsq}, the term $ \vb{t}_{K_k} $ denotes the vector of length $\lvert K_k \rvert$ which includes only the entries of $\vb{t}$ which have indicies appearing in the set $K_k$. Similarly $(\vb{A} \odot \vb{B})_{K_k}$ is the matrix where we restrict rows of $(\vb{A} \odot \vb{B})$ to only those which have indices appearing in the set.

Note $(\vb{A} \odot \vb{B})$ is full rank because $\vb{A}$ and $\vb{B}$ are assumed to be full rank (see, e.g., \cite{Berge2000}).  This implies the uncensored system is consistent with a unique solution. The difficulty in the censored case is that unobserved values in a particular column of the right hand side of \eqref{eqn:censored_lstsq} could force the discarding of rows of the matrix $(\vb{A} \odot \vb{B})$ that cause the system to become under-determined. 

That is, we must ensure that $(\vb{A} \odot \vb{B})_{K_k}$ has full column rank for each $k$. If we do not vary our sampling procedure from frontal slice to frontal slice, then this corresponds to sampling tubes of the original tensor of the form $\mathcal{T}_{i,j,:}$ and $K_k = K$ for all $k\in[n]$. In order to arrange the sampling to accomplish this, consider a $\vb{QR}$ with column-pivoting factorization of $(\vb{A}\odot \vb{B})^T$, see algorithm 5.4.1 in \cite{golubvanloan1996}). This produces factors and a permutation matrix $\vb{\Pi} \in \mathbb{R}^{n^2 \times n^2}$ such that $(\vb{A}\odot \vb{B})^T \vb{\Pi} = \vb{QR}$, where
\[
\vb{R} = \begin{bmatrix}
\vb{R}_1 & \vb{R}_2    
\end{bmatrix}
\]
and $\vb{R}_1 \in \mathbb{R}^{r\times r}$ is upper-triangular and non-singular. But this means that the first $r$ columns of $\vb{\Pi}$ select a set of columns of $(\vb{A}\odot \vb{B})^T$ which are linearly independent, so for each $\vb{\Pi}_{:,p}$, $p\in[r]$ we have that there exists some $q\in[n^2]$ such that $\vb{e}_{q} = \vb{\Pi}_{:,p}$, the corresponding column of the identity in $\mathbb{R}^{n^2\times n^2}$. Let $i = q \mod n, j = \lfloor \frac{q}{n} \rfloor$. Define then $\Omega_2$ all the tuples of the form $(i,j,:)$. That is, we can read off from $\vb{\Pi}$ which $r$ fibers of length $n$ to sample in order to ensure \eqref{eqn:censored_lstsq} is consistent for each $k\in[n]$. We have specified at most $nr$ new sample locations, and thus $\Omega = \Omega_1 \cup \Omega_2$, the set of all samples employed to complete the tensor $\mathcal{T}$ is at most $\lvert \Omega_1 \rvert + \lvert \Omega_2 \rvert \le C'_1s\mu_0 nr\log^2(r^2/\delta) + nr \le C_1s\mu_0 nr\log^2(r^2/\delta)$.

\begin{remark}
Note, $\vb{\Pi}$ is not unique, and indeed we could use different selections corresponding to other valid permutations from frontal slice to frontal slice to vary the sampling pattern when finding $\vb{C}$. Additionally, including more rows of $(\vb{A} \odot \vb{B})$ beyond the first $r$ as specified by $\vb{\Pi}$ may be numerically advantageous when computing $\vb{C}$.
\end{remark}

\RestyleAlgo{ruled}

\SetKwComment{Comment}{/* }{ */}

\begin{algorithm}[hbt!]
\SetKwComment{Comment}{\# }{}
\SetKwInOut{Input}{input}\SetKwInOut{Output}{output}%
\caption{Consistency Preserving Fiber Sampler}	\label{alg:fiber_sample}
 \Input
  {\par $\gamma \geq 1$, over-sample parameter
  \par $r$ rank parameter
   \par $\vb{A},\vb{B} \in \R^{n \times r}$, estimates for factor matrices
 }
\Output{ $\Omega_2$, set of sample locations}
Compute $\vb{QR}$ with column-pivoting, $\vb{QR} = \left(\vb{A} \odot \vb{B}\right)^T \vb{\Pi}$

\For{$p \in [\gamma r]$}{
$q \gets \text{nonzero}(\vb{\Pi}_{:,p}) $

$i \gets q \mod n$

$j \gets \lfloor \frac{q}{n} \rfloor$

Include $(i,j,k)$ in $\Omega_2$ for all $k\in[n]$
}
\end{algorithm}

\RestyleAlgo{ruled}

\SetKwComment{Comment}{/* }{ */}

\begin{algorithm}[hbt!]{
\SetKwComment{Comment}{\# }{}
\SetKwInOut{Input}{input}\SetKwInOut{Output}{output}%
\caption{Tensor Sandwich}	\label{alg:tensor_sandwich}
 \Input
  {$S\subseteq [n]$ slices to complete }
\Output{$\hat{T}$, completed tensor }
\Comment{Slice Complete Phase}
\For{$k\in S$}{
Use, e.g., \cite{Ward2015}, or algorithm 1 in \cite{krishnamurthy2014power} to complete $\calT_{:,:,k}$. Algorithm 1 in \cite{krishnamurthy2014power} uses at most $C'_1\mu_0 nr\log^2(r^2/\delta)$ adaptively chosen samples.
}
\Comment{Jennrich Complete Phase}

Generate random vectors $\vb{u},\vb{v}\in \mathbb{S}^{s-1}$;

$\vb{T}_{u} \gets \sum_{k\in S} u_{k} \vb{T}_{::k}$

$\vb{T}_{v} \gets \sum_{k\in S} v_{k} \vb{T}_{::k}$

Compute eigen-decomposition $\vb{A}\vb{\Lambda} \vb{A}^{-1} = \vb{T}_u (\vb{T}_u)^{\dagger}$

$\vb{B} \gets \vb{A}^{-1} \vb{T}_{u}$

\Comment{Censored Least Squares Phase}

$\Omega_2 \gets \text{fiber sampler} (\vb{A},\vb{B},r,\gamma$) 

$\vb{C}^T \gets \text{censored least squares }( (\vb{A}\odot  \vb{B}), \mathcal{T}_{i,j,k} \text{ s.t. } (i,j,k)\in \Omega_2)$

$\mathcal{\hat{T}} \gets [\![ \vb{A},\vb{B},\vb{C}]\!]$
}

\end{algorithm}

\section{Experiments} In this section, we show that tensors sampled according to our overall strategy and completed using Algorithm \eqref{alg:tensor_sandwich} support our theoretical findings. In particular we demonstrate that once sample complexity bounds are satisfied, we can achieve very precise levels of relative error by sampling what is overall a small percentage of the total tensor's entries. For example, in the rank 20 case in Figure \ref{fig:sample_study}, the median relative error is $0.000168$ having sampled $\frac{94445}{8000000}  \approx 1.1\%$ of the total entries. Our experiments also verify the linear dependence on rank in the sample complexity bound. Furthermore, in our second set of experiments we show that the method performs useful completion even in the presence of noise. The code and results from the numerical experiments is available at \url{https://github.com/cahaselby/TensorSandwich}.

In all our experiments, the three modes of the tensor have length $n=200$. The data is generated by drawing factor matrices $\vb{A},\vb{B},\vb{C} \in \mathbb{R}^{n\times r}$ with i.i.d standard Gaussian entries and then normalizing the columns. We then weight the components using quadratically decaying weights, i.e. \begin{equation} \label{eqn:data_gen}
    \calT = \sum_{i=1}^r \left(\frac{1}{i^2}\right) \vb{a}_i \circ \vb{b}_i \circ \vb{c}_i.
\end{equation}
Errors are averaged over ten independent trials.  

For each of our trials, $s=2$. After uniformly selecting two frontal slices to complete, within these slices we sample per Algorithm 1 in \cite{Ward2015} using a sample budget of $m$ samples per slice. Slices are then completed using the SDP formulation of the nuclear norm minimization solved via Douglas-Rachford splitting, see \cite{ocpb:16}. Convergence is declared once the primal residual, dual residual and duality gap are below $10^{-8}$ for each of the $s$ selected slices, or after 2500 iterations, whichever comes first. We note that the accuracy of this matrix completion step influences numerically what is achievable in terms of overall accuracy for the completed tensor, and that the error for completing these slices is compounded in the subsequent steps (even in the absence of noise). This explains the apparent ``leveling off'' of the relative error at about $10^{-4}$ even as sample complexity or signal-to-noise ratios increase. 

Following matrix completion on these frontal slices and estimation of $\vb{A}$ and $\vb{B}$, we select $\gamma r$ rows of $(\vb{A}\odot\vb{B})$ using Algorithm \ref{alg:fiber_sample} and the corresponding fibers of $\mathcal{T}$ to solve the censored least square problem. In all the experiments in this section, $\gamma=4$. This results in at most $\gamma n r$ new entries revealed in the original tensor. Total number of entries sampled then is always bounded by $sm + \gamma nr$ (some overlap of samples is possible in the two slices already completed). Note that for all experiments $0.035 nr \log^2(n) < m < 0.7 nr \log^2(n)$.

In Figure \ref{fig:sample_study}, we see a phase transition of our method. For a given rank, prior to a threshold, the error is dominated by the inability to accurately complete the $s$ slices. Once sufficient samples are obtained within these slices, completion reliably succeeds and accuracy approaches the limiting numerical accuracy inherited from the initial slice completion step.  

\begin{figure}[h]
\centering
\includegraphics[width=\columnwidth]{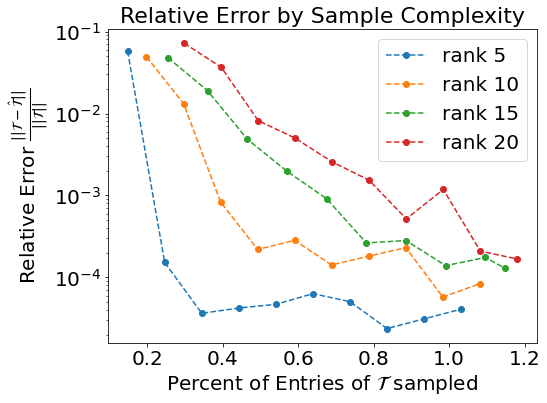}
\caption{Median relative error (log-scaled) of completed tensors of varying rank as sample complexity increases without noise. Each value is the median of ten trials.}
\label{fig:sample_study}
\end{figure}
In Figure \ref{fig:SNR_study}, we add mean-zero i.i.d. Gaussian noise to each entry in our tensor and perform the same completion strategy on the noisy tensor as described earlier in this section. For each trial, the noise tensor $\mathcal{N}$ is scaled to the appropriate signal-to-noise ratio, i.e. $\text{SNR} = 10\log_{10}\tfrac{\|\calT\|}{\|\mathcal{N}\|}$. Weights for each of the components are again set as per  \eqref{eqn:data_gen}. The sample complexity proportions are fixed with respect to slice completion versus fibers sampled to estimate $\vb{C}$ but do scale according to rank in order to facilitate comparison. In all cases the total number of sampled entries is between $0.27\%$ and $1.10\%$ of the total number of entries, depending on rank. 
\begin{figure}[h]
\centering
\includegraphics[width=\columnwidth]{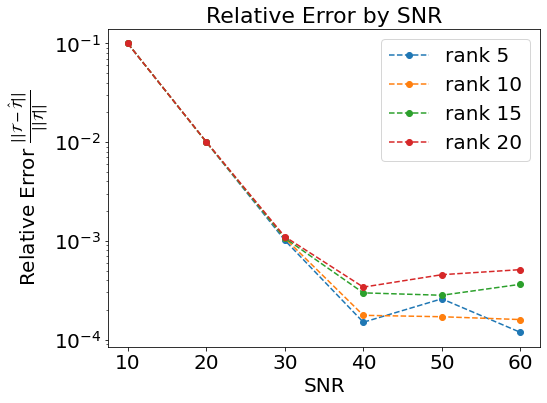}
\caption{Median relative error (log-scaled) of completed tensors of varying rank for different levels of white Gaussian noise. Sample complexity is less than 1.1\% for all trials, scaled linearly by rank.}
\label{fig:SNR_study}
\end{figure}

We noticed during these experiments that using masked alternating least squares alone (as in \cite{TOMASI2005163}) to complete our synthetic tensors (using either the Tensor Sandwich sample pattern, or an equivalent proportion of uniformly drawn samples) achieved at best relative errors of about 6\%. However, the estimate of the tensor as output by Algorithm \ref{alg:tensor_sandwich}, and its sample pattern can be used as an initialization to an iterative scheme like alternating least squares to further improve the accuracy of the completed tensor. In Figure \ref{fig:ts_als_compare}, we show three sets of related experiments:  the relative error at various ranks achieved by alternating least squares alone (AS), Tensor Sandwich alone (TS), or alternating least squares initialized by Tensor Sandwich (TS + ALS). In each trial, 100 iterations of alternating least squares are used, weights for components are set to decay according to \eqref{eqn:data_gen}, and each of the three variations for completing the tensor are used on the same data. We notice that at this level of sample complexity, for either a sample mask chosen uniformly at random or according to the adaptive scheme described in this work, ALS alone is limited. Combined with Tensor Sandwich, however, we can see roughly an order of magnitude improvement in relative error by performing a hundred iterations of ALS on the TS estimate. This shows Tensor Sandwich can be useful as an initialization strategy for other completion methods, either to save on run time by decreasing the total number of iterations needed, or to improve the accuracy of the final estimate.
\begin{figure}[h]
\centering
\includegraphics[width=\columnwidth]{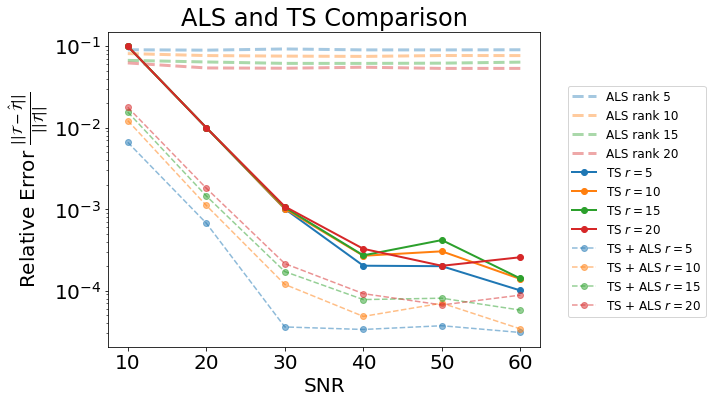}
\caption{Median relative error (log-scaled) of completed tensors of varying rank for different levels of white Gaussian noise. Sample complexity is less than 1.1\% for all trials, scaled linearly by rank. Tensors are completed using 100 iterations of alternating least squares (ALS), or Tensor Sandwich (TS), or 100 iterations of ALS initialized by Tensor Sandwich (TS+ALS).}
\label{fig:ts_als_compare}
\end{figure}

In a final set of experiments, we compare one of the existing (non-adaptive) completion methods discussed in the Related Work section to Tensor Sandwich. We use the same overall sample budget for Tensor Sandwich but instead use the Algorithm in \cite{moitra2018algorithmic} which we refer to as Tensor Complete (TC). The range of total values sampled is the same as in Figure \ref{fig:sample_study}. Implementation details make direct comparisons difficult.  However, the empirical findings summarized in Figure \ref{fig:ts_tc_compare} suggest that Tensor Sandwich can produce better estimates for high rank tensors than Tensor Complete can when sample budgets are limited. 
\begin{figure}[h]
\centering
\includegraphics[width=\columnwidth]{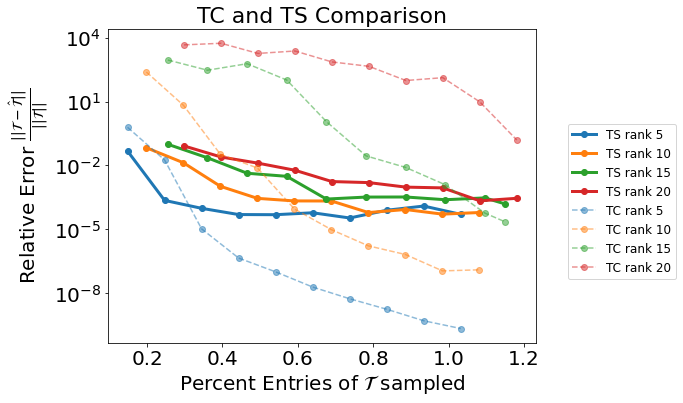}
\caption{Median relative error (log-scaled) of completed tensors of varying rank for different total number of entries revealed. Tensors are completed using  Tensor Sandwich (TS), or Tensor Complete as described in \cite{moitra2018algorithmic} without noise.}
\label{fig:ts_tc_compare}
\end{figure}

\section{Conclusions and Future Work}
We have presented an adaptive sampling approach for low CP-rank tensor completion which completes a CP-rank $r$ tensor of size $n \times n \times n$ using $\mathcal{O}(nr\log^2 r)$ samples with high probability. Our method significantly improves on the tensor completion result in \cite{krishnamurthy2013low} while only making a mild additional assumption on the third factor matrix. We also provided numerical experiments to demonstrate that a version of our tensor completion method is robust to noise empirically.

Several questions remain for future work. First, it would be interesting to prove rigorous guarantees for a non-adaptive version of our tensor completion algorithm using as few additional assumptions as possible.  Second, it is possible to extend our tensor completion method to higher-order tensors in several ways.  Choosing the best approach based on sampling complexity analysis would be of value.  Finally, it is always important to provide rigorous completion guarantees for the case where the observed entries are corrupted with some sort of noise.  Very little prior work on low-rank tensor completion provides any such theoretical guarantees making a treatment of this case an high-value goal.

\section*{Acknowledgements}
This work was supported in part by NSF DMS 2106472.

\bibliographystyle{IEEEtran}
\bibliography{ref}
\end{document}

%% file: macro.tex
\usepackage{url}
\usepackage{amsmath,amsthm,amssymb,amsbsy}
\usepackage{mathdots}
\usepackage{paralist}
\usepackage{xcolor}
\usepackage{color}
\usepackage{graphicx}
\usepackage{algorithm,algpseudocode}
\usepackage{comment}

\usepackage{fancyhdr}
\usepackage{cite}
\usepackage{cleveref}
\usepackage{enumerate}

\newtheorem{theorem}{Theorem}

\theoremstyle{remark}
\newtheorem{remark}{Remark}

\newcommand{\R}{\mathbb{R}}

\newcommand{\vct}[1]{\boldsymbol{#1}}
\newcommand{\mtx}[1]{\boldsymbol{#1}}

\newcommand{\eps}{\epsilon}

\newcommand{\calT}{\mathcal{T}}

\newcommand{\calP}{\mathcal{P}}

\newcommand{\ve}{\vct{e}}

\newcommand{\mA}{\mtx{A}}
\newcommand{\mB}{\mtx{B}}
\newcommand{\mC}{\mtx{C}}

\newcommand{\mT}{\mtx{T}}

\graphicspath{{./figs/}}

\newlength{\imgwidth}
\newboolean{twoColVersion}
\setboolean{twoColVersion}{false}
\newcommand{\twoCol}[2]{\ifthenelse{\boolean{twoColVersion}} {#1} {#2} }